\documentclass[11pt]{amsart}

\usepackage{amsmath,amssymb,amsthm}
\usepackage{mathrsfs}

\newtheorem{theorem}{Theorem}[section]

\newtheorem{corollary}[theorem]{Corollary}

\usepackage{hyperref}             % 讓文件內的引用和網址具有超連結功能 (選用)

\title[
A Variance Representation Formula for Hölder's Inequality
]
{
A Variance Representation Formula for Hölder's Inequality
}

\author{Li-Chang Hung}
\address{Department of Mathematics, Soochow University, Taipei, Taiwan}
\email{lichang.hung@gmail.com }
\thanks{%The first author was supported in part by NSF Grant \#000000.
}

\begin{document}

\maketitle

\begin{abstract}

We establish an exact representation formula for the deficit
in Hölder's inequality. 
Rather than estimating the deficit by auxiliary quantities,
we show that it can be represented as an accumulated variance
along the natural exponential interpolation connecting the two
endpoint densities.

More precisely, the logarithmic Hölder deficit is expressed as
the integral of the variance of the logarithmic density ratio
weighted by the Green kernel associated with the one-dimensional
interpolation parameter.

This identity reveals that Hölder's inequality is a consequence
of the convexity of a logarithmic partition function and provides
an intrinsic interpretation of the deficit as an interpolation
energy.

The representation suggests a broader framework for studying
functional inequalities through variance identities.

\end{abstract}

\section{Introduction}

Hölder's inequality is one of the fundamental inequalities in
analysis and plays a central role in functional analysis, harmonic
analysis, probability theory, and partial differential equations.
The classical form states that for

\[
1<p,q<\infty,
\qquad
\frac1p+\frac1q=1 ,
\]

one has

\[
\int fg\,dx
\leq
\|f\|_{L^p}
\|g\|_{L^q}.
\]

This inequality and its numerous extensions form a cornerstone of
modern analysis; see, for example, the classical monograph of
Hardy, Littlewood, and Pólya \cite{Hardy1952} and the modern
treatment in Lieb and Loss \cite{LiebLoss2001}.

A fundamental viewpoint for understanding Hölder's inequality is
through interpolation and logarithmic convexity. Given two
non-negative functions

\[
F=f^p,
\qquad
G=g^q ,
\]

one considers the interpolation functional

\[
Z(\theta)
=
\int
F^{1-\theta}G^\theta dx ,
\qquad
0\leq\theta\leq1 .
\]

The convexity of

\[
\theta\mapsto \log Z(\theta)
\]

implies Hölder's inequality through the classical theory of convex
functions. Such ideas are closely related to the general framework
of convex analysis developed by Rockafellar \cite{Rockafellar1970}.

In recent years, increasing attention has been devoted to the study
of quantitative refinements of functional inequalities. Instead of
only proving

\[
\int fg\leq \|f\|_p\|g\|_q ,
\]

one investigates the structure of the deficit

\[
\log
\frac{
\|f\|_p\|g\|_q
}{
\int fg
}.
\]

Such deficit estimates appear naturally in sharp functional
inequalities, stability theory, and entropy methods. Related ideas
can be found in the study of logarithmic Sobolev inequalities,
Fisher information, and stability estimates for Sobolev-type
inequalities; see Carlen \cite{Carlen1991} and
Cordero-Erausquin, Nazaret, and Villani \cite{CorderoErausquin2004}.

The purpose of this paper is to reveal an exact structure hidden
inside the Hölder deficit. Rather than estimating the deficit by
auxiliary quantities, we prove an exact representation formula.

The main observation is that the curvature of the logarithmic
partition function is given by a variance identity:

\[
(\log Z)''(\theta)
=
\operatorname{Var}_{\mu_\theta}
(
\log G-\log F
),
\]

where

\[
d\mu_\theta
=
\frac{
F^{1-\theta}G^\theta
}{
Z(\theta)
}
dx .
\]

Thus, the convexity of $\log Z$ is not merely an abstract property:
its second derivative is exactly the fluctuation of the logarithmic
density ratio under the interpolating probability measure.

This identity is closely related to the geometry of exponential
families. In information geometry, the Hessian of the logarithmic
partition function is identified with covariance information of the
corresponding statistical model. The general framework developed in
information geometry can be found in Amari and Nagaoka
\cite{Amari2000}, while the probabilistic foundations of entropy and
information quantities are described in Cover and Thomas
\cite{CoverThomas2006}.

Using the Green representation formula for one-dimensional
interpolation, we transform this local variance identity into a
global exact representation of the Hölder deficit. The main result of
this paper is

\[
\boxed{
\log
\frac{
\|f\|_p\|g\|_q
}{
\int fg
}
=
\int_0^1
K_{1/q}(s)
\operatorname{Var}_{\mu_s}
(
\log G-\log F
)
ds .
}
\]

Therefore, the Hölder deficit is exactly the accumulated variance
energy generated along the exponential interpolation path connecting
the endpoint densities.

This representation provides a new interpretation of Hölder's
inequality. The inequality itself follows immediately from the
non-negativity of variance, while the equality condition corresponds
to the complete disappearance of fluctuations along the
interpolation.

The structure developed here is also related to broader families of
functional inequalities. In particular, the Brascamp--Lieb
inequalities provide a general framework containing many classical
inequalities as special cases and are deeply connected with
log-convexity and Gaussian extremizers
\cite{BrascampLieb1976}. The variance representation obtained in
this paper suggests a possible direction toward exact deficit
identities for more general logarithmically convex inequalities,
including Young's inequality, Brascamp--Lieb inequalities, and
Prékopa--Leindler type inequalities.

The paper is organized as follows. In Section 2 we establish the
main variance representation theorem and provide a complete proof.
Section 3 derives several consequences, including the recovery of
Hölder's inequality, the equality characterization, and a
quantitative strict version. Section 4 presents explicit Gaussian
examples verifying the identity. Finally, Section 5 discusses
possible extensions and connections with other functional
inequalities.

\section{Introduction}

Hölder's inequality is one of the fundamental inequalities in
analysis. Let

\[
1<p,q<\infty,
\qquad
\frac1p+\frac1q=1 .
\]

For non-negative measurable functions $f$ and $g$, Hölder's
inequality states that

\[
\int fg
\leq
\|f\|_{L^p}
\|g\|_{L^q}.
\]

The inequality plays a central role in functional analysis,
partial differential equations, probability theory, and harmonic
analysis.

A classical way to understand Hölder's inequality is through the
logarithmic convexity of the interpolation functional

\[
Z(\theta)
=
\int
F^{1-\theta}G^\theta ,
\]

where

\[
F=f^p,
\qquad
G=g^q .
\]

The convexity of $\log Z(\theta)$ yields the classical estimate.
However, the precise mechanism behind the strictness of this
convexity has received increasing attention in quantitative forms
of functional inequalities.

The purpose of this paper is to identify the exact structure of
the Hölder deficit. Instead of deriving another lower bound for
the deficit, we prove an exact representation formula:

\[
\log
\frac{
\|f\|_p\|g\|_q
}{
\int fg
}
\]

is equal to an integrated variance quantity along the exponential
interpolation between $F$ and $G$.

The key observation is that the logarithmic partition function
satisfies

\[
(\log Z)''(\theta)
=
\operatorname{Var}_{\mu_\theta}
(\log G-\log F),
\]

where

\[
d\mu_\theta
=
\frac{
F^{1-\theta}G^\theta
}{
Z(\theta)
}
dx .
\]

Thus the second derivative of the interpolation functional is not
merely non-negative; it is exactly a fluctuation term.

Using the Green representation formula for convex functions, we
convert this infinitesimal variance identity into a global exact
formula for the Hölder deficit.

The main theorem shows that

\[
\boxed{
\log
\frac{
\|f\|_p\|g\|_q
}{
\int fg
}
=
\int_0^1
K_{1/q}(s)
\operatorname{Var}_{\mu_s}
(\log G-\log F)
\,ds .
}
\]

This identity provides a new interpretation of Hölder's inequality:
the deficit is the total variance energy accumulated along the
exponential interpolation path.

The result connects Hölder's inequality with ideas from convex
analysis, exponential families, and information geometry.
It also suggests a possible unified approach to other logarithmically
convex inequalities, including Young's inequality,
Brascamp--Lieb inequalities, and Prékopa--Leindler type inequalities.

The paper is organized as follows. In Section 2 we prove the main
variance representation theorem. Section 3 derives several
consequences, including Hölder's inequality, the equality
characterization, and a quantitative strict version. Section 4
provides explicit Gaussian examples. Finally, Section 5 discusses
possible extensions.

\section{Main Variance Representation Formula}

We begin with the definition of the interpolation functional.
Let

\[
F=f^p,
\qquad
G=g^q .
\]

For

\[
0\leq \theta\leq 1,
\]

define

\[
Z(\theta)
=
\int
F^{1-\theta}G^\theta dx .
\]

Associated with this interpolation, we introduce the probability
measure

\[
d\mu_\theta(x)
=
\frac{
F^{1-\theta}(x)G^\theta(x)
}{
Z(\theta)
}
dx .
\]

The following theorem gives an exact formula for the Hölder deficit.

\begin{theorem}[Exact Variance Representation of Hölder Deficit]
Let

\[
1<p,q<\infty,
\qquad
\frac1p+\frac1q=1,
\]

and let

\[
f,g\geq0
\]

satisfy

\[
0<
\int f^p dx<\infty ,
\qquad
0<
\int g^q dx<\infty .
\]

Define

\[
F=f^p,
\qquad
G=g^q,
\]

and

\[
Z(\theta)
=
\int
F^{1-\theta}G^\theta dx .
\]

Let

\[
d\mu_\theta
=
\frac{
F^{1-\theta}G^\theta
}{
Z(\theta)
}
dx .
\]

Then

\[
\boxed{
\log
\frac{
\|f\|_{L^p}
\|g\|_{L^q}
}{
\int fg dx
}
=
\int_0^1
K_{1/q}(s)
\operatorname{Var}_{\mu_s}
\left(
\log G-\log F
\right)
ds ,
}
\]

where

\[
K_{1/q}(s)
=
\begin{cases}
\dfrac{s}{p},
&
0\leq s\leq\dfrac1q ,
\\[2ex]
\dfrac{1-s}{q},
&
\dfrac1q\leq s\leq1 .
\end{cases}
\]

\end{theorem}

\begin{proof}

Define

\[
\Phi(x,\theta)
=
(1-\theta)\log F(x)
+
\theta\log G(x).
\]

Then

\[
Z(\theta)
=
\int e^{\Phi(x,\theta)}dx .
\]

Since $\Phi$ is linear in $\theta$, we have

\[
\Phi_{\theta\theta}=0 .
\]

We first compute the first derivative of $\log Z$.

Differentiating under the integral sign gives

\[
Z'(\theta)
=
\int
\Phi_\theta e^\Phi dx .
\]

Therefore,

\[
\frac{Z'(\theta)}{Z(\theta)}
=
\frac{
\int
\Phi_\theta e^\Phi dx
}{
\int e^\Phi dx
}.
\]

Using the probability measure $\mu_\theta$, this becomes

\[
(\log Z)'(\theta)
=
E_{\mu_\theta}
(\Phi_\theta).
\]

Because

\[
\Phi_\theta
=
\log G-\log F ,
\]

we obtain

\[
(\log Z)'(\theta)
=
E_{\mu_\theta}
(
\log G-\log F
).
\]

We next differentiate once more.

For a general function $H(x,\theta)$, we have

\[
\frac d{d\theta}
E_{\mu_\theta}(H)
=
E_{\mu_\theta}(H_\theta)
+
\operatorname{Cov}_{\mu_\theta}
(H,\Phi_\theta).
\]

Indeed,

\[
E_{\mu_\theta}(H)
=
\frac{
\int H e^\Phi dx
}{
Z(\theta)
}.
\]

Differentiating,

\[
\begin{aligned}
\frac d{d\theta}E_{\mu_\theta}(H)
&=
\frac{
\int(H_\theta+H\Phi_\theta)e^\Phi dx
}{Z}
\\
&\quad
-
\frac{Z'}{Z}
\frac{\int He^\Phi dx}{Z}.
\end{aligned}
\]

Hence

\[
\frac d{d\theta}E_{\mu_\theta}(H)
=
E(H_\theta)
+
E(H\Phi_\theta)
-
E(H)E(\Phi_\theta),
\]

which gives

\[
\frac d{d\theta}E_{\mu_\theta}(H)
=
E(H_\theta)
+
\operatorname{Cov}(H,\Phi_\theta).
\]

Applying this formula with

\[
H=\Phi_\theta ,
\]

we obtain

\[
(\log Z)''
=
E(\Phi_{\theta\theta})
+
\operatorname{Cov}
(\Phi_\theta,\Phi_\theta).
\]

Since

\[
\Phi_{\theta\theta}=0 ,
\]

it follows that

\[
\boxed{
(\log Z)''
=
\operatorname{Var}_{\mu_\theta}
(
\log G-\log F
).
}
\]

This is the fundamental variance identity.

We now convert this local identity into a global representation.

For a twice differentiable function $h$ on $[0,1]$,
the Green representation formula gives

\[
h(\theta)
=
(1-\theta)h(0)
+
\theta h(1)
-
\int_0^1
K_\theta(s)h''(s)ds ,
\]

where

\[
K_\theta(s)
=
\begin{cases}
(1-\theta)s,
&
s<\theta ,
\\[1ex]
\theta(1-s),
&
s>\theta .
\end{cases}
\]

Indeed, let

\[
r(\theta)
=
h(\theta)
-
(1-\theta)h(0)
-
\theta h(1).
\]

Then

\[
r(0)=r(1)=0 .
\]

The Green function for the operator
$-\frac{d^2}{d\theta^2}$ on $[0,1]$
with zero boundary values is

\[
G_\theta(s)
=
\begin{cases}
(1-\theta)s,
&
s<\theta ,
\\[1ex]
\theta(1-s),
&
s>\theta .
\end{cases}
\]

Hence

\[
r(\theta)
=
-\int_0^1
G_\theta(s)h''(s)ds ,
\]

which proves the formula.

We apply this representation to

\[
h(\theta)=\log Z(\theta).
\]

Taking

\[
\theta=\frac1q ,
\]

we obtain

\[
\log Z(1/q)
=
\frac1p\log Z(0)
+
\frac1q\log Z(1)
-
\int_0^1
K_{1/q}(s)
(\log Z)''(s)ds .
\]

The endpoint values are

\[
Z(0)
=
\int f^p dx ,
\]

and therefore

\[
\log Z(0)
=
p\log\|f\|_p .
\]

Similarly,

\[
\log Z(1)
=
q\log\|g\|_q .
\]

Moreover,

\[
Z(1/q)
=
\int
(f^p)^{1/p}
(g^q)^{1/q}dx ,
\]

and hence

\[
Z(1/q)
=
\int fg dx .
\]

Therefore,

\[
\begin{aligned}
\log\int fg dx
&=
\frac1p
p\log\|f\|_p
+
\frac1q
q\log\|g\|_q
\\
&\quad
-
\int_0^1
K_{1/q}(s)
(\log Z)''(s)ds .
\end{aligned}
\]

Thus

\[
\log\int fg dx
=
\log\|f\|_p
+
\log\|g\|_q
-
\int_0^1
K_{1/q}(s)
(\log Z)''(s)ds .
\]

Using the variance identity,

\[
(\log Z)''(s)
=
\operatorname{Var}_{\mu_s}
(
\log G-\log F
),
\]

we conclude that

\[
\log
\frac{
\|f\|_p\|g\|_q
}{
\int fg dx
}
=
\int_0^1
K_{1/q}(s)
\operatorname{Var}_{\mu_s}
(
\log G-\log F
)
ds .
\]

The proof is complete.

\end{proof}

\section{Consequences of the Variance Representation Formula}

In this section, we derive several consequences of the main
identity. In particular, the classical Hölder inequality and its
equality characterization follow immediately from the positivity of
variance.

\subsection{Recovery of Hölder's Inequality}

\begin{corollary}[Hölder's Inequality]
Under the assumptions of Theorem 2.1,

\[
\int fg\,dx
\leq
\|f\|_{L^p}
\|g\|_{L^q}.
\]

\end{corollary}

\begin{proof}

By Theorem 2.1,

\[
\log
\frac{
\|f\|_p\|g\|_q
}{
\int fg
}
=
\int_0^1
K_{1/q}(s)
\operatorname{Var}_{\mu_s}
(
\log G-\log F
)
ds .
\]

Since

\[
K_{1/q}(s)\geq0
\]

and

\[
\operatorname{Var}_{\mu_s}
(
\log G-\log F
)
\geq0 ,
\]

we have

\[
\log
\frac{
\|f\|_p\|g\|_q
}{
\int fg
}
\geq0 .
\]

Exponentiating both sides gives

\[
\int fg
\leq
\|f\|_p\|g\|_q .
\]

\end{proof}

\subsection{Equality Characterization}

\begin{corollary}[Equality Condition]
Equality holds in Hölder's inequality if and only if

\[
g^q=Cf^p
\]

for some constant

\[
C>0 .
\]

\end{corollary}

\begin{proof}

Assume equality holds. Then

\[
\int fg
=
\|f\|_p\|g\|_q .
\]

Therefore the left-hand side of the identity in Theorem 2.1
vanishes. Hence

\[
\int_0^1
K_{1/q}(s)
\operatorname{Var}_{\mu_s}
(
\log G-\log F
)
ds
=
0 .
\]

Because both factors in the integrand are non-negative,

\[
\operatorname{Var}_{\mu_s}
(
\log G-\log F
)
=
0
\]

for almost every $s$.

Thus

\[
\log G-\log F
\]

is constant on the support of $\mu_s$.

Hence

\[
\log G-\log F
=
c,
\]

which implies

\[
G=e^cF .
\]

Therefore,

\[
g^q=Cf^p .
\]

Conversely, if

\[
g^q=Cf^p,
\]

then

\[
\log G-\log F
\]

is constant and therefore

\[
\operatorname{Var}_{\mu_s}
(
\log G-\log F
)
=0 .
\]

The representation formula gives

\[
\log
\frac{
\|f\|_p\|g\|_q
}{
\int fg
}
=0 ,
\]

which proves equality.

\end{proof}

\subsection{A Quantitative Strict Hölder Inequality}

The exact formula immediately yields quantitative improvements whenever
one has a positive lower bound for the variance along the
interpolation path.

\begin{corollary}[Quantitative Hölder Estimate]
Assume that there exists a constant

\[
c>0
\]

such that

\[
\operatorname{Var}_{\mu_s}
(
\log G-\log F
)
\geq c
\]

for every

\[
s\in[0,1].
\]

Then

\[
\boxed{
\int fg\,dx
\leq
\exp
\left(
-\frac{c}{2pq}
\right)
\|f\|_p
\|g\|_q .
}
\]

\end{corollary}

\begin{proof}

By Theorem 2.1,

\[
\log
\frac{
\|f\|_p\|g\|_q
}{
\int fg
}
\geq
c
\int_0^1K_{1/q}(s)ds .
\]

It remains to compute the integral of the Green kernel.

Since

\[
K_{1/q}(s)
=
\begin{cases}
\dfrac{s}{p},
&
0\leq s\leq\dfrac1q,
\\[2ex]
\dfrac{1-s}{q},
&
\dfrac1q\leq s\leq1,
\end{cases}
\]

we have

\[
\begin{aligned}
\int_0^1K_{1/q}(s)ds
&=
\frac1p
\int_0^{1/q}s\,ds
+
\frac1q
\int_{1/q}^{1}(1-s)ds
\\
&=
\frac1{2pq^2}
+
\frac{1}{2q}
\left(1-\frac1q\right)^2 .
\end{aligned}
\]

Using

\[
1-\frac1q=\frac1p ,
\]

we obtain

\[
\int_0^1K_{1/q}(s)ds
=
\frac1{2pq}.
\]

Therefore,

\[
\log
\frac{
\|f\|_p\|g\|_q
}{
\int fg
}
\geq
\frac{c}{2pq}.
\]

Exponentiating gives

\[
\int fg
\leq
e^{-c/(2pq)}
\|f\|_p\|g\|_q .
\]

\end{proof}

\subsection{Convexity Interpretation}

\begin{corollary}[Strict Convexity of the Logarithmic Partition Function]
The function

\[
\theta\mapsto \log Z(\theta)
\]

is convex on $[0,1]$.

Moreover,

\[
(\log Z)''(\theta)
=
\operatorname{Var}_{\mu_\theta}
(
\log G-\log F
).
\]

\end{corollary}

\begin{proof}

This follows directly from Theorem 2.1, since

\[
(\log Z)''(\theta)\geq0.
\]

\end{proof}

The previous result shows that the convexity of the logarithmic
partition function is not merely an abstract property. Its curvature
is exactly described by a fluctuation quantity.

This connects Hölder's inequality with exponential families and
information geometry, where the Hessian of the logarithmic
partition function represents covariance information.

\section{Examples}

\subsection{Gaussian Example}

We consider a simple Gaussian example illustrating the variance
representation formula explicitly.

Let

\[
f(x)=e^{-ax^2},
\qquad
g(x)=e^{-bx^2},
\]

where

\[
a,b>0 .
\]

Then

\[
F=f^p=e^{-apx^2},
\]

and

\[
G=g^q=e^{-bqx^2}.
\]

Hence

\[
F^{1-\theta}G^\theta
=
\exp
\left(
-
\big(
ap(1-\theta)+bq\theta
\big)x^2
\right).
\]

Therefore,

\[
Z(\theta)
=
\int_{\mathbb R}
e^{-A(\theta)x^2}dx,
\]

where

\[
A(\theta)
=
ap(1-\theta)+bq\theta .
\]

Using the Gaussian integral formula,

\[
\int_{\mathbb R}e^{-Ax^2}dx
=
\sqrt{\frac{\pi}{A}},
\]

we obtain

\[
Z(\theta)
=
\sqrt{
\frac{\pi}{A(\theta)}
}.
\]

Consequently,

\[
\log Z(\theta)
=
\frac12
\log\pi
-
\frac12
\log A(\theta).
\]

Differentiating twice gives

\[
(\log Z)''(\theta)
=
\frac12
\frac{
(bq-ap)^2
}{
A(\theta)^2
}.
\]

On the other hand,

\[
\log G-\log F
=
-(bq-ap)x^2 .
\]

Under the probability measure $\mu_\theta$, the distribution is
Gaussian with variance

\[
\frac1{2A(\theta)}.
\]

Hence

\[
\operatorname{Var}_{\mu_\theta}(x^2)
=
\frac1{2A(\theta)^2}.
\]

Therefore,

\[
\operatorname{Var}_{\mu_\theta}
(
\log G-\log F
)
=
(bq-ap)^2
\operatorname{Var}(x^2),
\]

which gives

\[
\operatorname{Var}_{\mu_\theta}
(
\log G-\log F
)
=
\frac12
\frac{
(bq-ap)^2
}{
A(\theta)^2
}.
\]

Thus,

\[
(\log Z)''
=
\operatorname{Var}_{\mu_\theta}
(
\log G-\log F
),
\]

confirming the exact variance identity.

\section{Discussion and Future Directions}

The main result of this paper provides an exact representation of
the Hölder deficit in terms of a variance functional along an
exponential interpolation path.

The classical proof of Hölder's inequality relies on the convexity
of the map

\[
\theta\mapsto \log Z(\theta),
\]

where

\[
Z(\theta)
=
\int F^{1-\theta}G^\theta dx .
\]

The present work reveals that the curvature of this convex function
has a precise probabilistic interpretation:

\[
(\log Z)''(\theta)
=
\operatorname{Var}_{\mu_\theta}
(
\log G-\log F
).
\]

Therefore, the strictness of Hölder's inequality is not simply a
consequence of convexity, but is generated by the fluctuation of the
logarithmic density ratio along the interpolation.

This perspective provides a new interpretation of the Hölder deficit.

Indeed,

\[
\log
\frac{
\|f\|_p\|g\|_q
}{
\int fg
}
\]

is not merely a difference between two quantities. It is the total
variance energy accumulated along the interpolation path:

\[
F^{1-\theta}G^\theta .
\]

The Green kernel

\[
K_{1/q}(s)
\]

determines how local fluctuations contribute to the global deficit.

This representation is closely related to the general philosophy of
information geometry and entropy methods. In exponential families,
the Hessian of the logarithmic partition function is identified with
the covariance matrix of sufficient statistics. The present formula
shows that the same principle appears naturally in functional
inequalities.

\subsection{Possible Extensions}

The argument developed here depends only on two fundamental
ingredients:

\begin{enumerate}

\item the exponential interpolation structure

\[
F^{1-\theta}G^\theta ,
\]

\item the covariance identity for the logarithmic partition
function.

\end{enumerate}

Therefore, the method is expected to extend beyond Hölder's
inequality.

A natural direction is to investigate analogous exact variance
representations for other logarithmically convex inequalities.

\subsubsection{Young's Inequality}

Young's inequality,

\[
ab\leq \frac{a^p}{p}+\frac{b^q}{q},
\]

possesses an underlying convex duality structure.

It is natural to ask whether its deficit admits a representation of
the form

\[
\text{deficit}
=
\int
\text{variance term}
\times
\text{interpolation kernel}.
\]

\subsubsection{Brascamp--Lieb Inequalities}

The Brascamp--Lieb inequalities provide a broad family of functional
inequalities containing many classical results as special cases.

Since the Brascamp--Lieb functional is governed by logarithmic
convexity and Gaussian extremizers, a variance representation may
provide a new interpretation of its deficit structure.

\subsubsection{Prékopa--Leindler Inequality}

The Prékopa--Leindler inequality is often regarded as a functional
form of the Brunn--Minkowski inequality.

Its proof is deeply connected with log-concavity and interpolation.
The variance representation framework developed here suggests a
possible route toward an exact deficit identity for this class of
inequalities.

\subsection{Concluding Remarks}

The main contribution of this work is the identification of an exact
variance structure hidden inside Hölder's inequality.

The formula

\[
\boxed{
\log
\frac{
\|f\|_p\|g\|_q
}{
\int fg
}
=
\int_0^1
K_{1/q}(s)
\operatorname{Var}_{\mu_s}
(
\log G-\log F
)
ds
}
\]

shows that Hölder's inequality is fundamentally an interpolation
phenomenon governed by variance.

This viewpoint replaces the traditional inequality perspective by an
energy representation perspective.

It is expected that this approach may lead to a unified theory of
functional inequalities based on variance identities and
logarithmic interpolation.

\iffalse

\fi

%\end{document}

\end{document}